\documentclass{compositio}
\usepackage[T1]{fontenc}
\usepackage{lmodern,amsmath,mathtools}
\usepackage[hidelinks]{hyperref}
\hypersetup{pdftitle={Uniform positivity of the tau invariant of metrized graphs},
pdfsubject={Uniform positivity and counterexamples to the 1/108 bound},
pdfauthor={Ruihua Wang},
pdfkeywords={metrized graphs, tau invariant, tropical Jacobians, Euclidean lattices, effective resistance}}
\newtheorem{theorem}{Theorem}[section]

\newtheorem{lemma}[theorem]{Lemma}
\newtheorem{corollary}[theorem]{Corollary}
\theoremstyle{definition}
\newtheorem{definition}[theorem]{Definition}
\theoremstyle{remark}
\newtheorem{remark}[theorem]{Remark}
\newcommand{\R}{\mathbb R}
\newcommand{\Z}{\mathbb Z}
\newcommand{\E}{\mathbb E}
\newcommand{\ee}{\mathrm e}
\newcommand{\tr}{\operatorname{tr}}
\newcommand{\Cov}{\operatorname{Cov}}
\newcommand{\Vor}{\operatorname{Vor}}
\newcommand{\Jac}{\operatorname{Jac}}
\newcommand{\diag}{\operatorname{diag}}
\newcommand{\dd}{\,\mathrm d}
\newcommand{\ct}{c_{\tau}}
\newcommand{\cm}{c_{-}}
\allowdisplaybreaks[2]
\begin{document}
\title[Uniform positivity of the tau invariant]
{Uniform positivity of the tau invariant\\of metrized graphs}
% Author name and institutional email verified from the university's staff page.
\author{Ruihua Wang}
\email{ruih.wang@hotmail.com}
\address{Hainan Bielefeld University of Applied Sciences\\Danzhou, Hainan 571700\\China}
\classification{05C50 (primary), 14T20, 11H06, 14G40 (secondary)}
\keywords{metrized graphs, tau invariant, tropical Jacobians,
Euclidean lattices, effective resistance}
\begin{abstract}
We prove the Baker--Rumely conjecture that the tau invariant of a
metrized graph admits a positive lower bound proportional to its total
length, with an absolute constant. We also construct simple cubic
metrized graphs whose normalized tau invariants tend to
$59/7260<1/108$, disproving the proposed universal constant $1/108$.
The lower bound is independent of the genus, the number of edges,
and the distribution of edge lengths. Its proof combines a
second-moment inequality for Euclidean lattices with a partition of
the edge coordinates of a cycle lattice into three independent sets.
The counterexamples have only two edge lengths and admit an elementary
resistance calculation. Through the tropical moment identity, the
lower bound also gives a uniform estimate for the non-archimedean
terms in height formulas for Jacobians.
\end{abstract}
\maketitle

\section{Introduction}

Metrized graphs enter arithmetic geometry through local intersection
pairings. The capacity pairing of Chinburg--Rumely and the admissible
pairing of Zhang \cite{CR93,Zhang93} provide non-archimedean counterparts
of constructions on compact Riemann surfaces. Baker and Rumely
established the corresponding harmonic analysis in \cite{BR07}.
The tau invariant is the Dirichlet energy of effective resistance,
with the normalization
\begin{equation}\label{eq:intro-tau}
 \tau(\Gamma)=\frac14\int_\Gamma
       \left(\frac{\partial r(p,x)}{\partial x}\right)^2\dd x.
\end{equation}
Here $\Gamma$ is a finite connected metrized graph, $r(p,x)$ is
effective resistance, and $x$ is arclength on each edge. The value
is independent of the point $p$ \cite[Lemma~14.4]{BR07}.
We give the electrical definitions in Section~\ref{sec:preliminaries}.

Write $\ell(\Gamma)$ for the sum of the edge lengths. Both $\tau$
and $\ell$ scale linearly when all lengths are multiplied by the
same positive number. The extremal question is therefore to determine
\begin{equation}\label{eq:optimal-constant}
 \ct=\inf_{\substack{\Gamma\text{ finite connected metrized graph}\\
                    \ell(\Gamma)>0}}
                    \frac{\tau(\Gamma)}{\ell(\Gamma)}.
\end{equation}
Baker and Rumely asked whether $\ct>0$
\cite[Conjecture~14.5]{BR07}.

Earlier bounds depend on the complexity or connectivity of a model.
Baker--Rumely proved $\tau\geq\ell/(16m)$ for a model with $m$
edges \cite[(14.3)]{BR07}. Cinkir obtained
$\tau\geq\ell/(2(n+6))$ for a model with $n$ vertices, and proved
$\tau\geq\ell/108$ when the edge connectivity is at least six and
$\tau\geq\ell/300$ when it is five
\cite[Theorem~6.10]{CinkirConnectivity}. For a model with at least two vertices, the edge connectivity is
the least number of edges whose deletion disconnects it; for a
one-vertex model we set it to $+\infty$.
These estimates may be stronger on their respective classes, but
they do not give a positive constant for all graphs. Cinkir also
constructed families whose normalized tau invariants tend to $1/108$
\cite{CinkirFamilies}. Our construction shows that this limiting
value is not a universal lower bound, even for cubic graphs.

\begin{theorem}\label{thm:main}
Let $\cm$ be the unique number in $(0,1/12)$ satisfying
\begin{equation}\label{eq:constant}
 \cm(1-4\cm)^2=\frac4{27\pi^2\ee^2}.
\end{equation}
Then every finite connected metrized graph $\Gamma$ satisfies
\begin{equation}\label{eq:bridge-bound}
\tau(\Gamma)\geq\cm\bigl(\ell(\Gamma)-b(\Gamma)\bigr)+\frac{b(\Gamma)}4,
\end{equation}
where $b(\Gamma)$ is the sum of the lengths of its bridges.
\end{theorem}

With this theorem, we prove the Baker–Rumely conjecture as a corollary. 
\begin{corollary}
    For every finite connected metrized graph $\Gamma$, we have $\frac{\tau(\Gamma)}{\ell(\Gamma)}\geq\cm>0$.
\end{corollary}

Cinkir proposed the stronger inequality
$\tau(\Gamma)\geq\ell(\Gamma)/108$
\cite[Conjecture~1.5]{CinkirConnectivity}. We also give an explicit counterexample for this.

\begin{theorem}\label{thm:counter}There is an explicit family $(\Gamma_M)_{M\geq3}$ of simple cubic
metrized graphs, with edge lengths $1$ and $3$, for which
\begin{equation}\label{eq:main-limit}
 \lim_{M\to\infty}\frac{\tau(\Gamma_M)}{\ell(\Gamma_M)}
       =\frac{59}{7260}<\frac{1}{108}.
\end{equation}
\end{theorem}

The family in Theorem \ref{thm:counter} is
defined in Section~\ref{sec:counterexamples}. Its $M$th member has
$6M^2$ vertices and $9M^2$ edges. In particular, the proposed bound
fails within the class of simple cubic metrized graphs. The lower
constant is $\cm=0.0020654439\ldots$; the optimal value of $\ct$ remains open. An explicit finite
counterexample is given in Corollary~\ref{cor:finite-counterexample}.

The two parts of Theorem~\ref{thm:main} use the same electrical
quantities. If $e=uv$ has length $L_e$, its Foster coefficient is
$F_e=1-r(u,v)/L_e$. The standard edge formula writes tau as
$\frac1{12}\sum_eL_eF_e^2$ plus a nonnegative term. For the lower
bound, we control the reciprocals $L_e/F_e$ through the integer
cycle lattice. The estimate that makes this possible holds for
arbitrary Euclidean lattices.

\begin{theorem}\label{thm:intro-lattice}
Let $\Lambda$ be a full-rank lattice in a Euclidean space $W$, and let
$y_1,\ldots,y_k$ be linearly independent vectors of its dual lattice
$\Lambda^*$. Then
\begin{equation}\label{eq:intro-lattice}
 I(\Lambda)\geq\frac1{2\pi\ee}
                  \sum_{i=1}^k\frac1{\|y_i\|^2},
\end{equation}
where $I(\Lambda)$ is the average squared norm in the centered
Voronoi cell of $\Lambda$.
\end{theorem}

The definitions of the dual lattice and the Voronoi cell are given
in Definition~\ref{def:moment}. The dual vectors in the theorem need
neither be orthogonal nor extend to an integer basis. This flexibility
is needed for the edge coordinates of a graph. The proof controls
all principal minors of their covariance matrix and then turns these
determinant estimates into a trace bound. The resulting constant is
independent of the dimension and of the angles between the vectors.

For the application to graphs, the remaining issue is to account for
all edge coordinates using a bounded number of independent families.
After bridges and two-edge cuts have been removed by reductions that
preserve the relevant cycle lattice, Edmonds's partition theorem gives
three such families. This is where the combinatorics of cycle lattices
enters the proof; the lattice estimate itself does not require a graph.

The identity of de Jong--Shokrieh
\cite[Theorem~B]{JSMoments},
\[
 I\bigl(\Jac(\Gamma)\bigr)
       =\frac{\ell(\Gamma)}8-\frac{\tau(\Gamma)}2,
\]
converts the lattice estimate into the lower bound. Thus uniform
positivity of tau is also a uniform gap between the normalized
moment of a tropical Jacobian and $1/8$. This identity retains the
connection with arithmetic geometry: tropical moments occur in
the non-archimedean terms in the Faltings--theta height formula of
de Jong--Shokrieh \cite{JSHeights}. Remark~\ref{rem:arithmetic}
records the resulting bound, with a coefficient independent of the
genus and the reduction type.
For the upper bound on $\ct$, the symmetries of $\Gamma_M$ permit
an exact computation of the two Foster coefficients. The remaining
term is bounded by half the vertex diameter and is negligible
relative to the total length.

Section~\ref{sec:preliminaries} fixes the definitions and records the
published identities used in both arguments. We prove the lattice
estimate in Section~\ref{sec:lattice}, deduce the uniform lower bound
in Section~\ref{sec:lower}, and construct the counterexamples in
Section~\ref{sec:counterexamples}.

\section{Electrical networks and cycle lattices}\label{sec:preliminaries}

\subsection{Metrized graphs and effective resistance}
A metrized graph is obtained by identifying endpoints of finitely
many closed intervals of positive lengths. A finite model
$G=(V,E)$ records these intervals as edges and their endpoints as
vertices; loops and multiple edges are allowed. Subdivision of an
edge, with its length split accordingly, changes the model but not
the metrized graph. Unless stated otherwise, models are connected.
We write $n=|V|$, $m=|E|$, and
\[
 \ell=\ell(\Gamma)=\sum_{e\in E}L_e,
 \qquad g=m-n+1.
\]
The integer $g$ is the first Betti number of the graph. The
one-vertex graph without edges is allowed, with $\ell=\tau=0$.
All sums over $E$ count each unoriented edge once.

Choose one orientation for each edge. Denote by $\delta_e$ and
$\delta_v$ the coordinate vectors in $\R^E$ and $\R^V$,
respectively. The incidence map and energy inner product are
\begin{equation}\label{eq:incidence-energy}
 \partial\delta_e=\delta_v-\delta_u\quad(e:u\to v),
 \qquad
 \langle j,k\rangle_L=\sum_{e\in E}L_ej_ek_e.
\end{equation}
A loop has zero incidence column. The component $(\partial j)_v$
is the net inflow of the current $j$ at $v$.

\begin{definition}[(Effective resistance and tau)]\label{def:resistance}
For vertices $u,v$ of a connected model, define
\begin{equation}\label{eq:resistance}
 r(u,v)=\min\left\{\sum_{e\in E}L_ej_e^2:
                    j\in\R^E,\ \partial j=\delta_v-\delta_u\right\}.
\end{equation}
For arbitrary points of $\Gamma$, subdivide at those points before
applying this definition. With $p\in\Gamma$ fixed, choose a model containing $p$ as a vertex
and define
\begin{equation}\label{eq:tau}
 \tau(\Gamma)=\frac14\sum_{e\in E}
             \int_0^{L_e}\left(\frac{\mathrm d}{\mathrm dx}r(p,x)\right)^2\dd x.
\end{equation}
In the integral, $x$ is the point at arclength $x$ along the edge $e$.
\end{definition}

The feasible affine space in \eqref{eq:resistance} is nonempty, and
the positive definite energy has a unique minimizer.
At a subdivision vertex, conservation forces equal currents in the
two subedges, whose energies add. Resistance is therefore unchanged
by further subdivisions. Formula~\eqref{eq:tau} is independent of
$p$ by \cite[Lemma~14.4]{BR07}. Sending unit current along a path
also gives $r(u,v)$ at most the length of that path.

\begin{definition}[(Foster coefficients)]\label{def:foster}
For an edge $e=uv$ of length $L_e$, put
\begin{equation}\label{eq:foster-definition}
 F_e=1-\frac{r(u,v)}{L_e},\qquad
 A_G=\sum_{e\in E}L_eF_e^2,\qquad
 B_{G,p}=\sum_{e=uv\in E}\frac{(r(p,u)-r(p,v))^2}{L_e}.
\end{equation}
We call $F_e$ the Foster coefficient of $e$.
\end{definition}
For a loop $F_e=1$, and for a bridge $F_e=0$. Every other edge
satisfies $0<F_e<1$. Unlike $\tau$, the separate quantities $A_G$
and $B_{G,p}$ depend on the chosen model.

\subsection{Euclidean lattices and tropical Jacobians}
A full-rank lattice in a Euclidean space $W$ is the integer span
of a real basis of $W$. Its dual is taken with respect to the given
inner product, not to an arbitrarily chosen coordinate pairing.

\begin{definition}[(Dual lattice, Voronoi cell, and second moment)]\label{def:moment}
For a full-rank Euclidean lattice $\Lambda\subset W$, set
\begin{align}
 \Lambda^*&=\{y\in W:\langle y,z\rangle\in\Z
                                      \text{ for all }z\in\Lambda\},\label{eq:dual}\\
 \Vor(\Lambda)&=\{x\in W:\|x\|\leq\|x-z\|
                                      \text{ for all }z\in\Lambda\},\label{eq:voronoi}\\
 I(\Lambda)&=\frac{1}{\operatorname{vol}(\Vor(\Lambda))}
                  \int_{\Vor(\Lambda)}\|x\|^2\dd x.\label{eq:moment-definition}
\end{align}
The set $\Vor(\Lambda)$ is the centered Voronoi cell. Volume and
integration refer to Euclidean Lebesgue measure on $W$. In dimension
zero we set $I(\{0\})=0$.
\end{definition}

Thus the Voronoi cell consists of the points for which $0$ is a
nearest lattice point; ties occur on its boundary. It is a bounded,
centrally symmetric convex set whose lattice translates tile $W$
up to boundaries of measure zero. Equivalently, it contains one
representative of almost every class modulo $\Lambda$. A uniform
point in the cell induces Haar probability measure on $W/\Lambda$,
that is, the translation-invariant probability measure. These are the usual Voronoi
conventions of \cite[Section~2]{JSMoments}. For example,
\[
 \Vor(a\Z)=[-a/2,a/2],\qquad I(a\Z)=\frac{a^2}{12}\quad(a>0).
\]
In particular, our moment is the full mean squared norm, with no
factor $1/2$. If $X$ is uniform in the cell, then
$I(\Lambda)=\E\|X\|^2$ and $\E X=0$.

\begin{definition}[(Cycle lattice and tropical Jacobian)]\label{def:cycle}
For an oriented model $G=(V,E)$, let
\begin{equation}\label{eq:cycle}
 H_G=\ker(\partial:\R^E\to\R^V),\qquad
 \Lambda_G=H_G\cap\Z^E,
 \qquad \Jac(\Gamma)=H_G/\Lambda_G.
\end{equation}
Both $H_G$ and its integer cycle lattice $\Lambda_G$ carry the
edge-length inner product \eqref{eq:incidence-energy}. The quotient,
with this metric, is the tropical Jacobian. We write
$I_G=I(\Lambda_G)=I(\Jac(\Gamma))$.
\end{definition}

An element of $H_G$ is a circulation: its net inflow is zero at every
vertex. The fundamental cycles associated to a spanning tree form
an integer basis of $\Lambda_G$, so it has full rank $g$ in $H_G$.
Changing edge orientations or subdividing edges gives an isometric
cycle lattice. In particular, $I_G$ depends only on $\Gamma$.

\subsection{The identities used below}

\begin{theorem}[(de Jong--Shokrieh)]\label{thm:electrical}
For the quantities in Definitions~\ref{def:resistance}--\ref{def:cycle},
and a model containing $p$ as a vertex,
\begin{equation}\label{eq:edge}
 \tau(\Gamma)=\frac{A_G}{12}+\frac{B_{G,p}}4,
 \qquad
 I_G=\frac{\ell(\Gamma)}8-\frac{\tau(\Gamma)}2.
\end{equation}
\end{theorem}
\begin{proof}
These are \cite[Theorems~11.2 and~11.4]{JSMoments}, respectively,
in the normalizations above. Loops may either be retained or treated
by subdivision; their contributions are $L_e/12$ to tau and $L_e/12$
to the cycle-lattice moment.
\end{proof}

We also record the finite-dimensional electrical notation. Put
\[
 D_L=\diag(L_e),\qquad \Delta=\partial D_L^{-1}\partial^T.
\]
The symmetric matrix $\Delta$ is the conductance Laplacian. Its
kernel consists of constant vertex vectors. Its pseudoinverse
$\Delta^+$ is zero on the constants and is the inverse of $\Delta$
on their orthogonal complement, for the ordinary Euclidean inner
product on $\R^V$. Let $P_G$ denote orthogonal projection from
$(\R^E,\langle\ ,\ \rangle_L)$ onto $H_G$.

\begin{theorem}[(Kirchhoff's projection formula and Foster's identity)]\label{thm:kirchhoff}
With these conventions, we have
\begin{align}
 P_G&=\mathrm{Id}-D_L^{-1}\partial^T\Delta^+\partial,
 &r(u,v)&=(\delta_v-\delta_u)^T\Delta^+(\delta_v-\delta_u),\label{eq:projection}\\
 \sum_{e\in E}F_e&=g=m-n+1.&&\label{eq:foster}
\end{align}
\end{theorem}
\begin{proof}
The projection formula is \cite[Theorem~7.8 and~(7-1)]{JSRayleigh}.
The resistance formula follows from \cite[Proposition~5.2,
Remark~6.3(ii), and Example~6.5]{JSRayleigh}. Foster's identity is
stated in \cite[Section~7D(iii)]{JSRayleigh}.
\end{proof}

For the counterexamples we shall also use the vertex-transitive
case of Cinkir's Laplacian formula. A model is vertex-transitive
if its length-preserving automorphisms act transitively on its vertices.

\begin{theorem}[(Cinkir, vertex-transitive case)]\label{thm:transitive}
Let $G$ be a simple connected model whose length-preserving
automorphisms act transitively on its $n$ vertices. Then, for
any vertex $p$,
\begin{equation}\label{eq:transitive}
 \tau(\Gamma)=\frac1{12}\sum_{e\in E}L_eF_e^2
              +\frac1{2n}\sum_{v\in V}r(p,v).
\end{equation}
\end{theorem}
\begin{proof}
Apply \cite[Theorem~1.1]{CinkirLaplacian}. In the notation of
Section~\ref{sec:preliminaries}, transitivity makes the diagonal of
$\Delta^+$ constant, say equal to $d$. The cited formula becomes
$\tau=A_G/12+\tr(\Delta^+)/n=A_G/12+d$. Since $\Delta^+$ has
zero row sums, \eqref{eq:projection} gives
$\sum_v r(p,v)=2nd$, which is \eqref{eq:transitive}.
\end{proof}

Let $D_V(G)=\max_{u,v\in V}d_\Gamma(u,v)$ be the vertex diameter,
where $d_\Gamma$ is shortest-path distance measured using the edge
lengths. Since $0\leq r(p,v)\leq d_\Gamma(p,v)$, the preceding
theorem gives, for such a model of positive total length,
\begin{equation}\label{eq:transitive-diameter}
 \frac{A_G}{12\ell(\Gamma)}
 \leq\frac{\tau(\Gamma)}{\ell(\Gamma)}
 \leq\frac{A_G}{12\ell(\Gamma)}+\frac{D_V(G)}{2\ell(\Gamma)}.
\end{equation}

Theorem~\ref{thm:electrical} yields the two inequalities used in the
lower-bound argument:
\begin{equation}\label{eq:basic-estimates}
 \tau(\Gamma)\geq\frac{A_G}{12},
 \qquad I_G\leq\frac{\ell(\Gamma)}8.
\end{equation}
We shall keep the exact value of $I_G$ when deriving the constant
$\cm$.

\section{A moment bound from independent dual vectors}\label{sec:lattice}

The aim of this section is Theorem~\ref{thm:intro-lattice}. Recall from
Definition~\ref{def:moment} that $\Lambda^*$ consists of the vectors
having integral pairing with $\Lambda$, and that $I(\Lambda)$ is the
average of $\|x\|^2$ over the centered Voronoi cell.
Integrality implies that each independent collection of dual
coordinates sends a uniform point of this cell to a uniform point
of a quotient torus. We use this observation to bound the principal
minors of a covariance matrix. The matrix lemma below turns these
determinant bounds into the trace estimate needed to recover
$I(\Lambda)$. Keeping all coordinate subsets is what allows the
argument to accommodate nonorthogonal dual vectors.

\subsection{Two covariance estimates}
For a random column vector $Z$ with finite second moment, write
$\Cov Z=\E[(Z-\E Z)(Z-\E Z)^T]$. Here $\E$ denotes expectation,
$T$ denotes transpose, and $\tr$ denotes trace. A density $f$ is
always taken with respect to Lebesgue measure, so that
$\mathbb P(Z\in D)=\int_D f(x)\dd x$ for every measurable set $D$.
Positive definite matrices are understood to be real and symmetric.

\begin{lemma}\label{lem:matrix}
Let \(k\geq1\), and let \(B,Q\) be real positive definite \(k\)-by-\(k\) matrices.
Suppose \(\det Q_{SS}\geq1\) for every nonempty set
\(S\subseteq\{1,\ldots,k\}\), where \(Q_{SS}\) is the principal
submatrix indexed by \(S\). Then
\begin{equation}\label{eq:matrix}
 \tr(B^{-1}Q)\geq\sum_{i=1}^k\frac1{B_{ii}}.
\end{equation}
\end{lemma}
\begin{proof}
Simultaneously permute the rows and columns of $B$ and $Q$ so that
$a_i=B_{ii}^{-1}$ satisfies $a_1\geq\cdots\geq a_k>0$.
Both sides of \eqref{eq:matrix} are unchanged. The assumption on
\emph{all} principal minors ensures that the determinant conditions
remain valid after this permutation.

Since $Q$ is positive definite, successive elimination gives
\[
 Q=M\diag(q_1,\ldots,q_k)M^T,
 \qquad q_i>0,
\]
where $M$ is lower triangular with diagonal entries one. Write $m_i$
for its $i$th column and $\delta_i$ for the $i$th coordinate vector.
Then $\delta_i^Tm_i=1$. Applying Cauchy--Schwarz to
$B^{1/2}\delta_i$ and $B^{-1/2}m_i$ gives
\[
 1\leq(\delta_i^TB\delta_i)(m_i^TB^{-1}m_i),
 \qquad m_i^TB^{-1}m_i\geq a_i.
\]
As $Q=\sum_iq_im_im_i^T$, it follows that
\[
 \tr(B^{-1}Q)=\sum_iq_im_i^TB^{-1}m_i
                 \geq\sum_i a_iq_i.
\]

The individual $q_i$ need not be at least one; it is their initial
products that are controlled. Indeed, put
\[
 F_j=\log\det Q_{[j],[j]},\qquad [j]=\{1,\ldots,j\},\qquad F_0=0.
\]
The leading $j$-by-$j$ block of $M$ has determinant one, so
$F_j=\sum_{i=1}^j\log q_i\geq0$. Summation by parts now yields
\[
 \sum_{i=1}^k a_i\log q_i
 =a_kF_k+\sum_{i=1}^{k-1}(a_i-a_{i+1})F_i\geq0.
\]
Using $q_i\geq1+\log q_i$ and the preceding trace estimate, we get
\[
 \tr(B^{-1}Q)\geq\sum_i a_iq_i
       \geq\sum_i a_i+\sum_i a_i\log q_i
       \geq\sum_i a_i.
\]
\end{proof}

\begin{lemma}\label{lem:density}
Let \(Z\) be a random vector in \(\R^j\) with finite second moment,
positive definite covariance matrix \(C\), and Lebesgue density
\(f\leq1\) almost everywhere. Then
\begin{equation}\label{eq:det-cov}
 \det C\geq(2\pi\ee)^{-j}.
\end{equation}
\end{lemma}
\begin{proof}
Put \(Z_0=Z-\E Z\). Since
\(\E(Z_0^TC^{-1}Z_0)=\tr(C^{-1}C)=j\), Jensen's inequality and the
density bound give
\begin{align*}
 \exp(-j/2)
 &\leq\E\exp\bigl(-Z_0^TC^{-1}Z_0/2\bigr)\\
 &\leq\int_{\R^j}\exp\bigl(-w^TC^{-1}w/2\bigr)\dd w
 =(2\pi)^{j/2}\sqrt{\det C}.
\end{align*}
Squaring and rearranging proves the assertion.
\end{proof}

\subsection{Proof of Theorem \ref{thm:intro-lattice}}
We will use the lattice moment estimate to prove the theorem.
\begin{proof}[Proof of Theorem~\ref{thm:intro-lattice}]
The case $k=0$ is immediate. Assume $k\geq1$, and choose orthonormal
coordinates on $W$. Let $X$ be uniform on $\Vor(\Lambda)$ and set
$K=\Cov X$. Central symmetry gives $\E X=0$, so that
\[
 \tr K=\E\|X\|^2=I(\Lambda).
\]
Let $Y$ be the matrix with columns $y_1,\ldots,y_k$, and put
\[
 Z=Y^TX,\qquad B=Y^TY,\qquad C=Y^TKY=\Cov Z.
\]
Here $B$ is the Gram matrix of the chosen dual vectors. Both $B$
and $C$ are positive definite: $Y$ has full column rank, and the
Voronoi cell contains a nonempty open subset of $W$, so no nonzero
linear functional of $X$ is almost surely constant.

We first control the covariance of each coordinate subfamily of $Z$.
Fix a nonempty set $S\subseteq\{1,\ldots,k\}$ and consider
\[
 T_S:W\longrightarrow\R^S,
 \qquad T_S(x)=(\langle y_i,x\rangle)_{i\in S}.
\]
The condition $y_i\in\Lambda^*$ says precisely that
$T_S(\Lambda)\subseteq\Z^S$. Thus $T_S$ induces a homomorphism
\[
 \overline T_S:W/\Lambda\longrightarrow\R^S/\Z^S.
\]
It is surjective: for a representative $t\in\R^S$ of any target
class, linear independence of the $y_i$ gives an $x\in W$ with
$T_S(x)=t$. This argument does not require
$T_S(\Lambda)=\Z^S$.

Since $\Vor(\Lambda)$ is a fundamental domain, $X\bmod\Lambda$
has Haar probability distribution on $W/\Lambda$. The image of this
measure under $\overline T_S$ is translation invariant: every target
translation lifts, by surjectivity, to a translation of $W/\Lambda$.
It is therefore Haar probability measure on $\R^S/\Z^S$.
Consequently $Z_S=T_S(X)$ is uniform modulo $\Z^S$.

The random vector $Z_S$ has a Lebesgue density $f_S$. To see this,
choose coordinates on $W$ adapted to $\ker T_S$ and integrate the
constant density of $X$ along the fibers of $T_S$; Fubini's theorem
gives the desired marginal density. Reducing $Z_S$ modulo $\Z^S$
periodizes this density. More explicitly, for a measurable
$D\subseteq[0,1)^S$,
\[
 \mathbb P(Z_S\bmod\Z^S\in D)
   =\sum_{n\in\Z^S}\int_D f_S(t+n)\dd t.
\]
Uniformity on the torus implies
\[
 \sum_{n\in\Z^S}f_S(t+n)=1
 \quad\text{for almost every }t\in[0,1)^S.
\]
Every summand is nonnegative. It follows, on each integer translate
of the unit cube and hence on all of $\R^S$, that $f_S\leq1$
almost everywhere. Lemma~\ref{lem:density} now gives
\[
 \det C_{SS}\geq(2\pi\ee)^{-|S|}.
\]

Apply Lemma~\ref{lem:matrix} to $B$ and $Q=(2\pi\ee)C$.
The determinant just obtained supplies its hypothesis for every
principal submatrix, while $B_{ii}=\|y_i\|^2$. We obtain
\[
 (2\pi\ee)\tr\bigl((Y^TY)^{-1}Y^TKY\bigr)
     \geq\sum_{i=1}^k\frac1{\|y_i\|^2}.
\]
The matrix
\[
 P=Y(Y^TY)^{-1}Y^T
\]
is the orthogonal projection onto $\operatorname{span}(y_1,\ldots,y_k)$.
Cyclicity of trace identifies the trace in the preceding inequality
with $\tr(KP)$. In an orthonormal basis adapted to the range of $P$,
this is the sum of the first $k$ diagonal entries of $K$. Each such
entry is a variance, so
\[
 \tr(KP)\leq\tr K=I(\Lambda).
\]
Combining the last two estimates proves \eqref{eq:intro-lattice}.
\end{proof}

\section{The uniform lower bound}\label{sec:lower}

\subsection{Independent edge coordinates}
Fix a connected model $G=(V,E)$, with $n=|V|$ and $m=|E|$.
In this section $H_G=\ker\partial$
is its cycle space, $\Lambda_G=H_G\cap\Z^E$ its integer cycle
lattice, and $I_G$ its moment. All norms on $H_G$ are taken in the
edge-length inner product. Recall that $D_L=\diag(L_e)$ and
$P_G$ is orthogonal projection onto $H_G$ for this inner product.

\begin{lemma}\label{lem:coordinates}
For each edge $e$, the vector
\[
 y_e=P_G(D_L^{-1}\delta_e)\in H_G
\]
represents its coordinate functional:
\begin{equation}\label{eq:coordinate-norm}
 \langle y_e,x\rangle_L=x_e\quad(x\in H_G),
 \qquad y_e\in\Lambda_G^*,\qquad
 \|y_e\|_L^2=\frac{F_e}{L_e}.
\end{equation}
For $S\subseteq E$, let $c(G\setminus S)$ be the number of connected
components after deleting the edges of $S$ and retaining all vertices.
Then
\begin{equation}\label{eq:rank}
 \dim\operatorname{span}\{y_e:e\in S\}
       =|S|-c(G\setminus S)+1.
\end{equation}
In particular, $(y_e)_{e\in S}$ is linearly independent exactly when
$G\setminus S$ is connected.
\end{lemma}
\begin{proof}
For $x\in H_G$, orthogonal projection is self-adjoint and fixes $x$.
Therefore
\[
 \langle y_e,x\rangle_L
 =\langle D_L^{-1}\delta_e,x\rangle_L=x_e.
\]
When $x\in\Lambda_G$, its coordinates are integers, which proves
$y_e\in\Lambda_G^*$. In particular, taking $x=y_e$ shows that
$\|y_e\|_L^2$ is the $e$th coordinate of $y_e$.
Writing $b_e=\partial\delta_e$ and using \eqref{eq:projection}, we get
\[
 \|y_e\|_L^2
 =\langle D_L^{-1}\delta_e,P_G(D_L^{-1}\delta_e)\rangle_L
 =\frac1{L_e}-\frac{b_e^T\Delta^+b_e}{L_e^2}
 =\frac{F_e}{L_e}.
\]
This calculation includes a loop: then $b_e=0$ and $F_e=1$.

For the rank assertion, consider the restriction map
\[
 T_S:H_G\longrightarrow\R^S,\qquad x\longmapsto(x_e)_{e\in S}.
\]
The coordinate functionals of $T_S$ are represented by the vectors
$y_e$. The Riesz identification between $H_G$ and its dual therefore
gives
$\operatorname{rank}T_S=\dim\operatorname{span}\{y_e:e\in S\}$.
A circulation is in $\ker T_S$ precisely when it is supported on
$G\setminus S$. For a graph with $n$ vertices, $m-|S|$ edges and
$c=c(G\setminus S)$ connected components, the incidence map has
rank $n-c$. Its cycle space consequently has dimension
$(m-|S|)-n+c$. Rank--nullity gives
\[
 \operatorname{rank}T_S
 =(m-n+1)-\bigl((m-|S|)-n+c\bigr)=|S|-c+1.
\]
The family indexed by $S$ is independent exactly when this rank is
$|S|$, or equivalently when $c=1$.
\end{proof}

Theorem \ref{thm:intro-lattice} now gives
\begin{equation}\label{eq:independent-sum}
 \sum_{e\in S}\frac{L_e}{F_e}\leq2\pi\ee\,I_G
 \quad\text{if }G\setminus S\text{ is connected}.
\end{equation}
No bridge belongs to such an $S$, so every denominator is positive.

\begin{theorem}[(Edmonds)]\label{thm:edmonds}
Let $(z_e)_{e\in E}$ be a finite indexed family of vectors, and let
$k\geq1$ be an integer. The set $E$ can be partitioned into $k$
sets indexing linearly independent subfamilies if and only if
\[
 |S|\leq k\dim\operatorname{span}\{z_e:e\in S\}
 \qquad(S\subseteq E).
\]
Empty parts are allowed.
\end{theorem}
\begin{proof}
This is the matroid partition theorem
\cite[Theorem~1]{Edmonds}, applied to the matroid of linear independence
of the given vector family.
\end{proof}

A two-edge cut is a pair of edges whose simultaneous deletion
disconnects a graph, although neither edge is a bridge.

\begin{lemma}\label{lem:partition}
If $G$ has no bridges and no two-edge cuts, its edges admit a
partition $E=S_1\sqcup S_2\sqcup S_3$ such that $G\setminus S_i$
is connected for every $i$.
\end{lemma}
\begin{proof}
Fix $S\subseteq E$ and put $c=c(G\setminus S)$. If $c=1$, the
rank in \eqref{eq:rank} is $|S|$. If $c\geq2$, every component
has at least three edges to the other components. Counting these
edges at both ends gives $3c\leq2|S|$, whence
\[
 \dim\operatorname{span}\{y_e:e\in S\}
 =|S|-c+1\geq\frac{|S|}{3}+1.
\]
Apply Theorem~\ref{thm:edmonds} with $k=3$ and use
Lemma~\ref{lem:coordinates}. For a one-vertex graph, $c=1$ for
every $S$; loops cause no exception.
\end{proof}

\subsection{A reduction preserving the cycle lattice}
We use contraction only for non-loop edges: it identifies the two
endpoints and removes the edge, retaining any loops thereby created.
The following reduction is needed because the individual edge term
$A_G=\sum_eL_eF_e^2$ is not preserved by subdivision.

\begin{lemma}\label{lem:reduction}
Let $b=b(\Gamma)$ be the total bridge length. There is a connected
metrized graph $\Gamma_0$, possibly a single vertex, whose model has
no bridges or two-edge cuts, such that
\begin{equation}\label{eq:reduction}
 \ell(\Gamma_0)=\ell(\Gamma)-b,
 \qquad \tau(\Gamma)=\tau(\Gamma_0)+\frac b4.
\end{equation}
\end{lemma}
\begin{proof}
First contract all bridges. A circulation has zero coordinate on a
bridge: sum the conservation equations over either component of its
complement. Thus restriction to the remaining coordinates identifies
the real cycle spaces and the integer cycle lattices, preserving the
edge-length inner product. The resulting moment is unchanged, while
total length decreases by $b$. Recalling
\[
 I_G=\frac{\ell(\Gamma)}8-\frac{\tau(\Gamma)}2,
\]
we see that tau decreases by $b/4$.

It remains to remove two-edge cuts in the resulting bridgeless graph.
Suppose that $\{e,f\}$ is such a cut. Its deletion leaves exactly
two components. Indeed, if there were at least three, the connected
graph formed by contracting those components would have only two
edges and hence a bridge, contrary to the bridgeless assumption.
Call the two components $U$ and $V$. Orient $e$ from $U$ to $V$
and $f$ from $V$ to $U$. Summing conservation over $U$ gives
\[
 x_e=x_f\qquad(x\in H_G).
\]
Contract $f$ and give the surviving edge $e$ length $L_e+L_f$;
all other surviving lengths remain unchanged.

We spell out the identification of cycle spaces under this operation.
Restriction of a circulation to the edges other than $f$ is a
circulation after contraction. Conversely, let $x'$ be a circulation
on the contracted graph, and let $u,v$ be the former endpoints of
$f$. At all vertices other than $u,v$, the unchanged coordinates
already satisfy conservation. At $u$ and $v$ their net inflows are
opposite, because conservation holds at the merged vertex. There is
therefore a unique choice of the $f$ coordinate which cancels both
net inflows. If $x'$ has integer coordinates, so does this added
coordinate. The restriction map is consequently an isomorphism both
of real cycle spaces and of integer cycle lattices. Its inverse is
a circulation on the original graph and thus satisfies $x_f=x_e$.
The energy identity
\[
 L_ex_e^2+L_fx_f^2=(L_e+L_f)x_e^2
\]
now proves that the isomorphism is an isometry.

The total length has not changed, and neither has the lattice
moment. The moment identity therefore shows that tau has not changed.
No surviving edge becomes a bridge: every such edge lies on a cycle
before the operation, hence has a nonzero coordinate on some
circulation, whose restriction remains a circulation afterward.
Repeating the operation terminates, since it removes one edge at
each step. Loops are retained; in particular, an entire cycle may
become a single loop. The graph finally obtained is $\Gamma_0$,
and the asserted formulas follow.
\end{proof}

\subsection{Proof of Theorem \ref{thm:main}}
\begin{proof}[Proof of ~Theorem~\ref{thm:main}]
First suppose that $\Gamma$ has positive length and that its model
has no bridges or two-edge cuts. In particular, every Foster
coefficient $F_e=1-r(u,v)/L_e$ is positive. In the following
calculation, $\ell=\sum_eL_e$ is the total length,
$A_G=\sum_eL_eF_e^2$ is the edge term in the tau formula, and
$I_G=I(\Lambda_G)$ is the moment of the cycle lattice.
Summing
\eqref{eq:independent-sum} over the three parts in
Lemma~\ref{lem:partition} yields
\begin{equation}\label{eq:reciprocal-bound}
 \mathcal R_G:=\sum_{e\in E}\frac{L_e}{F_e}\leq6\pi\ee I_G.
\end{equation}
The reciprocal sum $\mathcal R_G$ measures how small the Foster
coefficients can be in aggregate. To relate it to the quadratic
sum $A_G$, apply H\"older's inequality with exponents $3$ and $3/2$:
\begin{equation}\label{eq:holder}
 \ell^3
 =\left[\sum_e(L_eF_e^2)^{1/3}(L_e/F_e)^{2/3}\right]^3
 \leq\left(\sum_eL_eF_e^2\right)
       \left(\sum_eL_e/F_e\right)^2
 =A_G\mathcal R_G^2.
\end{equation}
Set $t=\tau(\Gamma)/\ell(\Gamma)$. The exact moment identity and
\eqref{eq:reciprocal-bound} imply
\[
 0<\mathcal R_G\leq\frac{3\pi\ee}{4}\ell(1-4t).
\]
Since $\mathcal R_G>0$, this inequality also gives $1-4t>0$.
Consequently we may divide by $(1-4t)^2$. Using successively
$\tau\geq A_G/12$, \eqref{eq:holder}, and the displayed bound
on $\mathcal R_G$, we find
\[
 t\geq\frac{A_G}{12\ell}
   \geq\frac{\ell^2}{12\mathcal R_G^2}
   \geq\frac{4}{27\pi^2\ee^2(1-4t)^2}.
\]
Multiplication by $(1-4t)^2$ gives
\begin{equation}\label{eq:cubic-inequality}
 t(1-4t)^2\geq\frac4{27\pi^2\ee^2}.
\end{equation}
To extract the lower bound, let $f(u)=u(1-4u)^2$. Its derivative is
$f'(u)=(1-4u)(1-12u)$, so $f$ is strictly increasing on
$[0,1/12]$, with values $0$ and $1/27$ at the endpoints.
Since $0<4/(27\pi^2\ee^2)<1/27$, equation~\eqref{eq:constant}
has a unique root $\cm$ in $(0,1/12)$. If $0\leq t<\cm$, then
$f(t)<f(\cm)$, contrary to \eqref{eq:cubic-inequality}. This
argument does not require an upper bound $t\leq1/12$: any larger
value of $t$ already exceeds $\cm$. Thus
$\tau(\Gamma)\geq\cm\ell(\Gamma)$.

For an arbitrary graph, apply Lemma~\ref{lem:reduction}. If
$\Gamma_0$ has positive length, the preceding estimate gives
\[
 \tau(\Gamma)=\tau(\Gamma_0)+\frac b4
      \geq\cm(\ell-b)+\frac b4.
\]
If $\Gamma_0$ is a single vertex, the same equality gives
$\tau(\Gamma)=\ell/4$. This proves \eqref{eq:bridge-bound} and
also covers the graph with no edges, for which $\ell=\tau=0$.
Finally, $\cm<1/12<1/4$, so
\[
 \cm(\ell-b)+\frac b4
 =\cm\ell+\left(\frac14-\cm\right)b\geq\cm\ell.
\]
Taking the infimum of $\tau(\Gamma)/\ell(\Gamma)$ over graphs
of positive length gives the lower bound on the optimal value $\ct$.
\end{proof}

\begin{remark}[(Local terms in arithmetic height formulas)]\label{rem:arithmetic}
\widowpenalty=10000
Let $C$ be a smooth
projective geometrically connected curve of genus at least two over
a number field $k$, with semistable reduction, and let $J$ be its
Jacobian. For each finite place $v$, let $\Gamma_v$ be the dual graph of the
geometric special fiber of the minimal regular model of $C$ at $v$,
with each edge assigned length one. Its vertices correspond to
irreducible components and its edges to nodes.
Recall that $I(\Jac(\Gamma_v))$ is the average squared norm in the
Voronoi cell of its cycle lattice, and that $b(\Gamma_v)$ is the total
length of its bridges. The moment identity and
\eqref{eq:bridge-bound} give
\[
 0\leq I(\Jac(\Gamma_v))
 \leq\left(\frac18-\frac{\cm}{2}\right)
       \bigl(\ell(\Gamma_v)-b(\Gamma_v)\bigr).
\]
The height formula of de Jong--Shokrieh
\cite[Theorems~A and~B, and~(1.11)]{JSHeights} therefore gives the
following bound for its total non-archimedean contribution:
\[
 \frac1{[k:\mathbb Q]}\sum_{v\text{ finite}}
        I(\Jac(\Gamma_v))\log Nv
 \leq
 \frac{1}{[k:\mathbb Q]}\left(\frac18-\frac{\cm}{2}\right)
 \sum_{v\text{ finite}}
      \bigl(\ell(\Gamma_v)-b(\Gamma_v)\bigr)\log Nv,
\]
where $Nv$ is the cardinality of the residue field. This gives a
genus-independent bound for the finite-place correction in the
comparison of the stable Faltings height of $J$ with the N\'eron--Tate
height of a symmetric theta divisor, with the archimedean terms retained.
\end{remark}

\section{A periodic family below \texorpdfstring{$1/108$}{1/108}}\label{sec:counterexamples}

The lower-bound argument controlled the Foster coefficients without
computing them. We now choose a family for which they have only two
values, both explicit. Vertex transitivity then controls the other
term in the edge formula.

\subsection{The construction}
Let $R_M=\Z/M\Z$ for an integer $M\geq3$, and let $\mathbf e_i$
be the $i$th coordinate vector of $R_M^3$. Define
\[
 V_s(M)=\{x\in R_M^3:x_1+x_2+x_3=s\},\qquad s\in\{0,1\}.
\]
The periodic hexagonal graph $H_M$ has vertex set $V_0(M)\sqcup
V_1(M)$ and edges $\{x,x+\mathbf e_i\}$ for $x\in V_0(M)$ and
$i\in\{1,2,3\}$. Initially give these edges length one.

\begin{definition}[(The graph $\Gamma_M$)]\label{def:example}
The vertex set of $\Gamma_M$ is
\[
 V(\Gamma_M)=(V_0(M)\sqcup V_1(M))\times\{1,2,3\}.
\]
For each $x\in V_0(M)\sqcup V_1(M)$, join the three vertices
$(x,1),(x,2),(x,3)$ by all three triangle edges, each of length one.
For each $x\in V_0(M)$ and $i\in\{1,2,3\}$, add a connecting
edge of length three from $(x,i)$ to $(x+\mathbf e_i,i)$.
\end{definition}

Thus each vertex of $H_M$ has been replaced by a triangle, with its
three original incident edges attached at the three separate ports.
At a vertex $(x,i)$ there are two triangle edges and exactly one
connecting edge. The graph is therefore cubic. It is simple because
the three neighbors $x+\mathbf e_i$ of a vertex of $V_0(M)$ are
distinct, and connecting edges join different triangles.

The two vectors $\mathbf e_1-\mathbf e_3$ and
$\mathbf e_2-\mathbf e_3$ form a basis of the $R_M$-module $V_0(M)$:
every $x\in V_0(M)$ has the unique expression
\[
 x=x_1(\mathbf e_1-\mathbf e_3)
      +x_2(\mathbf e_2-\mathbf e_3).
\]
The move $x\mapsto x+\mathbf e_i-\mathbf e_3$, for $i=1,2$,
is realized by the two-edge path
$x,x+\mathbf e_i,x+\mathbf e_i-\mathbf e_3$ in $H_M$.
Thus all vertices of $V_0(M)$ lie in one component, and every vertex
of $V_1(M)$ is adjacent to this component. Replacing a vertex by a
connected triangle preserves connectedness, so $\Gamma_M$ is connected.

There are $M^2$ elements in each $V_s(M)$, since the first two
coordinates determine the third. Consequently $H_M$ has $2M^2$
vertices and $3M^2$ edges. Its replacement has $2M^2$ triangles,
hence $6M^2$ triangle edges of length one and $3M^2$ connecting edges
of length three. In particular,
\begin{equation}\label{eq:example-counts}
 |V(\Gamma_M)|=6M^2,\qquad |E(\Gamma_M)|=9M^2,
 \qquad \ell(\Gamma_M)=15M^2.
\end{equation}

Translations by $V_0(M)$, permutations of the three coordinates,
and $x\mapsto\mathbf e_1-x$ preserve adjacency in $H_M$.
They act transitively on its directed edges. Their lifts to $\Gamma_M$ are
\[
 (x,i)\longmapsto(x+a,i),\qquad
 (x,i)\longmapsto(\sigma x,\sigma(i)),\qquad
 (x,i)\longmapsto(\mathbf e_1-x,i),
\]
where $a\in V_0(M)$ and $\sigma$ is a permutation of $\{1,2,3\}$,
acting on coordinates by $\sigma\mathbf e_i=\mathbf e_{\sigma(i)}$.
Each map preserves the distinction between triangle edges and
connecting edges, and hence preserves lengths. For the last map,
the image of the connecting edge indexed by $(x,i)$ is the connecting
edge indexed by $(\mathbf e_1-x-\mathbf e_i,i)$, with the opposite
orientation. Translations are transitive on each $V_s(M)$,
coordinate permutations permute the three ports at $0$, and the
last map interchanges $V_0(M)$ and $V_1(M)$. It follows that the
action on vertices of $\Gamma_M$ is transitive. The same maps show
that all triangle edges lie in one orbit and all connecting edges
in another. They cannot lie in the same orbit because their lengths
are different.

\subsection{Proof of Theorem \ref{thm:counter}}
Write $F_t(M)$ and $F_c(M)$ for the Foster coefficients of a triangle
edge and a connecting edge of $\Gamma_M$, respectively.
Foster's identity \eqref{eq:foster} on the edge-transitive unit graph
$H_M$ gives a common edge resistance $1-\beta_M$, where
\[
 \beta_M=\frac{M^2+1}{3M^2}.
\]
Replace each unit triangle electrically by a star whose arms have
length $1/3$. Here is the energy calculation, in the convention of
Definition~\ref{def:resistance}. Fix currents $h_1,h_2,h_3$ entering
the three terminals, with $h_1+h_2+h_3=0$. If $a,b,c$ are the
cyclically oriented triangle currents, their conservation equations
are
\[
 a-c=h_1,\qquad b-a=h_2,\qquad c-b=h_3.
\]
Substitute $b=a+h_2$ and $c=a-h_1$. Completing the square gives
\[
 a^2+b^2+c^2
 =3\left(a-\frac{h_1-h_2}{3}\right)^2
       +\frac{h_1^2+h_2^2+h_3^2}{3}.
\]
Its minimum is precisely the energy of the star, whose arm currents
are $h_1,h_2,h_3$. For any fixed currents in the rest of the graph,
we may therefore minimize inside each triangle without changing
the minimum total energy. Effective resistances between the original
vertices are preserved. This also applies when a source or sink is
at a terminal: its prescribed current is included in $h_i$.
In the resulting network the star centers form a copy of $H_M$,
with each old edge replaced by three successive segments of lengths
$1/3,3,1/3$, of total length $11/3$.

Recall that the Foster coefficient is $F_e=1-r(u_e,v_e)/L_e$.
On the unit graph $H_M$, let $R$ be the resistance between the
endpoints after one edge has been deleted. This deletion leaves
a connected graph: edge transitivity makes the Foster coefficients
equal to the positive number $\beta_M$, whereas a bridge has
coefficient zero. The deleted edge, of resistance one, is in
parallel with the remaining network, so
\[
 1-\beta_M=\frac{R}{1+R},\qquad
 R=\frac{1-\beta_M}{\beta_M}.
\]
Now delete a connecting edge of $\Gamma_M$ and perform the star
replacements. Its two endpoint ports are joined to the corresponding
star centers by arms of length $1/3$. Between the centers lies
$H_M$ with one edge deleted and all remaining edge lengths scaled
by $11/3$. The two arms are in series with this network. Hence the
resistance between the ports, with the connecting edge absent, is
\[
 \rho_M=\frac23+\frac{11}{3}\frac{1-\beta_M}{\beta_M}.
\]
Restoring the connecting edge puts a resistance of three in parallel
with $\rho_M$. Its effective endpoint resistance is
$3\rho_M/(3+\rho_M)$, so its Foster coefficient is
\begin{equation}\label{eq:Fc}
 F_c(M)=\frac{3}{3+\rho_M}
       =\frac{9\beta_M}{11}
       =\frac3{11}(1+M^{-2}).
\end{equation}
Foster's identity on $\Gamma_M$ now yields
\[
 6M^2F_t(M)+3M^2F_c(M)=3M^2+1,
\]
and hence
\begin{equation}\label{eq:Ft}
 F_t(M)=\frac13+\frac{\beta_M}{11}
       =\frac4{11}+\frac1{33M^2}.
\end{equation}
The normalized edge contribution is therefore
\begin{equation}\label{eq:CM}
 C_M:=\frac{A_{\Gamma_M}}{12\ell(\Gamma_M)}
     =\frac{2F_t(M)^2+3F_c(M)^2}{60}
     \longrightarrow\frac{59}{7260}.
\end{equation}
Here $A_{\Gamma_M}=\sum_eL_eF_e^2$ is the edge sum from
Definition~\ref{def:foster}, for the model just constructed.

\begin{proof}[Proof of ~Theorem~\ref{thm:counter}]
Recall from \eqref{eq:transitive} that, for a vertex-transitive
model with $n$ vertices, the part of tau beyond $A_G/12$ is half
the average resistance from a fixed vertex:
\[
 \tau(\Gamma)-\frac{A_G}{12}
       =\frac1{2n}\sum_{v\in V}r(p,v).
\]
Its normalized contribution is at most $D_V(G)/(2\ell(\Gamma))$,
where $D_V(G)$ is the diameter of the vertex set in the metric of
the graph. For our family, total length is quadratic in $M$.
It is therefore enough to give a diameter bound linear in $M$;
no individual long-range resistance needs to be computed.

We now bound $D_V(\Gamma_M)$. On $H_M$, at most $\lfloor M/2\rfloor$
moves in each of the directions $\mathbf e_1-\mathbf e_3$ and
$\mathbf e_2-\mathbf e_3$, with either sign, connect two vertices
of $V_0(M)$. Each move uses two edges. The same holds on $V_1(M)$;
passing between the two parts needs at most one more edge.
Thus any two vertices of $H_M$ are joined by at most $2M+1$ edges.
A path of $k$ old edges lifts to $\Gamma_M$ with length at most
$3k+(k+1)=4k+1$, including the initial and final triangle traversals.
The vertex diameter of $\Gamma_M$ is consequently at most $8M+5$.

Using \eqref{eq:example-counts} and \eqref{eq:transitive-diameter}, we obtain
\begin{equation}\label{eq:sandwich}
 C_M\leq\frac{\tau(\Gamma_M)}{\ell(\Gamma_M)}
        \leq C_M+\frac{8M+5}{30M^2}.
\end{equation}
Together with \eqref{eq:CM}, this proves \eqref{eq:main-limit}.
Since
\[
 \frac1{108}-\frac{59}{7260}=\frac{37}{32670}>0,
\]
it also gives the upper bound on $\ct$ and completes the proof of
Theorem~\ref{thm:counter}.\end{proof}

The diameter estimate already yields a finite counterexample with an
explicit margin. No numerical approximation of effective resistance
is needed for the following conclusion.

\begin{corollary}\label{cor:finite-counterexample}
For the graphs $\Gamma_M$ defined in this section and every $M\geq300$,
\begin{equation}\label{eq:finite-gap}
 \frac{\tau(\Gamma_M)}{\ell(\Gamma_M)}
       <\frac1{108}-\frac{19}{540000}.
\end{equation}
In particular, $\Gamma_{300}$ is a counterexample to the universal
$1/108$ bound.
\end{corollary}
\begin{proof}
Equations~\eqref{eq:Fc}--\eqref{eq:Ft} imply
\[
 C_M\leq\frac{59}{7260}(1+M^{-2})^2.
\]
If $M\geq10$, use $59/7260<1/123$ and $(101/100)^2<41/40$ to get
$C_M<1/120$. For $M\geq300$,
\[
 \frac{8M+5}{30M^2}\leq\frac{481}{540000},
 \qquad
 \frac1{120}+\frac{481}{540000}
       =\frac1{108}-\frac{19}{540000}.
\]
Combining this with \eqref{eq:sandwich} proves the assertion.
\end{proof}

\begin{acknowledgements}
The author thanks DeepSeek and ChatGPT for assistance with
proof strategies, calculations, and exposition. Responsibility for
the mathematical content rests with the author.
\end{acknowledgements}

\end{document}